\documentclass{amsart}

\usepackage{amsmath,amsthm,mathrsfs,amssymb,graphicx,needspace}
\usepackage[all]{xy}
\usepackage{tikz}
\usetikzlibrary{shapes,snakes}
\usepackage[pdfborder={0 0 0},backref=false,colorlinks=true,linktocpage=true]{hyperref}

\newtheorem{theorem}{Theorem}[section]

\newtheorem{lemma}[theorem]{Lemma}
\newtheorem{corollary}[theorem]{Corollary}

\theoremstyle{definition}
\newtheorem{definition}[theorem]{Definition}

\newtheorem{innercase}{Case}
\newenvironment{case}[1]{\renewcommand\theinnercase{#1}\innercase}{\endinnercase}

\newcommand{\N}{\mathbb{N}}

\newcommand{\res}{\mathbin{\upharpoonright}}
\newcommand{\range}{\operatorname{ran}}

\renewcommand{\mod}{\text{ }\textrm{mod}\text{ }}

\newcommand{\sequence}[1]{\langle #1 \rangle}
\newcommand{\pair}[1]{\langle #1 \rangle}

\newcommand{\forces}{\Vdash}

\newcommand{\0}{\mathbf{0}}

\newcommand{\RCA}{\mathsf{RCA}_0}

\newcommand{\D}{\mathsf{D}}

\newcommand{\RT}{\mathsf{RT}}
\newcommand{\SRT}{\mathsf{SRT}}

\newcommand{\COH}{\mathsf{COH}}

\newcommand{\In}{0}

\begin{document}

\title{Cohesive avoidance and arithmetical sets}

\author{Damir D. Dzhafarov}
\address{Department of Mathematics\\
University of Notre Dame\\
Notre Dame, Indiana 46556, U.S.A.}
\email{ddzhafar@nd.edu}


\thanks{The author was partially supported by an NSF Postdoctoral Fellowship and by the John Templeton Foundation. He is most grateful to R.~A.~Shore and T.~A.~Slaman, conversations with whom, while visiting the Institute for Mathematical Sciences at the National University of Singapore in the summer of 2011, inspired the results in this article. He is also grateful to D.~R.~Hirschfeldt and J.~L.~Hirst for commenting on this article, especially to the former for pointing out that the proof of Theorem \ref{theorem_main}, originally formulated only for $\Delta^0_2$ operators, in fact works for arbitrary arithmetical operators.}

\maketitle

\begin{abstract}
An open question in reverse mathematics is whether the cohesive principle, $\COH$, is implied by the stable form of Ramsey's theorem for pairs, $\SRT^2_2$, in $\omega$-models of $\RCA$. One typical way of establishing this implication would be to show that for every sequence $\vec{R}$ of subsets of $\omega$, there is a set $A$ that is $\Delta^0_2$ in $\vec{R}$ such that every infinite subset of $A$ or $\overline{A}$ computes an $\vec{R}$-cohesive set. In this article, this is shown to be false, even under far less stringent assumptions: for all natural numbers $n \geq 2$ and $m < 2^n$, there is a sequence $\vec{R} = \sequence{R_0,\ldots,R_{n-1}}$ of subsets of $\omega$ such that for any partition $A_0,\ldots,A_{m-1}$ of $\omega$ arithmetical in $\vec{R}$, there is an infinite subset of some $A_j$ that computes no set cohesive for $\vec{R}$. This complements a number of previous results in computability theory on the computational feebleness of infinite sets of numbers with prescribed combinatorial properties. The proof is a forcing argument using an adaptation of the method of Seetapun showing that every finite coloring of pairs of integers has an infinite homogeneous set not computing a given non-computable set.
\end{abstract}

\section{Introduction}

There is an intricate relationship between the computational strength of sets of natural numbers and their combinatorial properties, and its exploration occupies a prominent role within computability theory. (We refer the reader to Soare \cite{Soare-1987} for a general introduction to the subject.) As a case in point, a number of results have been concerned with the computational feebleness of infinite subsets of a given set, as in the following examples:

\begin{theorem}[Soare \cite{Soare-1969}] There is an infinite set with no infinite subset of strictly higher Turing degree.
\end{theorem}

\begin{theorem}[Kjos-Hanssen \cite{Kjos-Hanssen-2009}, Theorem 3.8]
There is an infinite Martin-L\"{o}f random set with an infinite subset that computes no Martin-L\"{o}f random set.
\end{theorem}

\begin{theorem}[Dzhafarov and Jockusch \cite{DJ-2009}, Lemma 3.2]
If $A_0,\ldots,A_{m-1}$ partition $\omega$ and $C$ is any non-computable set, then there is an infinite subset of some $A_j$ that does not compute $C$.
\end{theorem}

\noindent The property of not computing the set $C$ above was called ``cone avoidance'' in \cite{DJ-2009}, as it can be described as avoiding the upper cone above $C$ under Turing reducibility. In this article, we shall be interested in a property that can, by analogy, be called ``cohesive avoidance''.

A set $S$ is said to be \emph{cohesive} for a sequence $\vec{R} = \sequence{R_i : i \in \omega}$, or \emph{$\vec{R}$-cohesive}, if it is infinite and for each $i$, either $S \cap R_i$ or $S \cap \overline{R_i}$ is finite. When $\vec{R}$ contains all computably enumerable (c.e.)~sets, the notion of cohesiveness is related to that of maximality (see \cite{Soare-1987}, Sections X.3 and XI.2), which has been of interest in the study of the lattice of c.e.\ sets under inclusion. Various classes of cohesive sets have been investigated, as have the Turing degrees of their members, e.g., in \cite{Cooper-1972}, \cite{Jockusch-1973}, \cite{JS-1993}, and \cite{HJ-1999}.

Every infinite set is, of course, cohesive for itself. More generally, given a finite sequence $\vec{R} = \sequence{R_0,\ldots,R_{n-1}}$, it is easy to build $2^n$ many $\vec{R}$-computable sets, each infinite subset of each of which is $\vec{R}$-cohesive: namely, the $2^n$ Boolean combinations of $R_0,\ldots,R_{n-1}$. Our main result is that the same effect cannot be achieved with fewer than $2^n$ many sets, even if we allow these to be merely arithmetical, rather than computable, in $\vec{R}$, and even if we ask that each infinite subset of each $A_j$ merely compute, rather than be, an $\vec{R}$-cohesive set.

\begin{theorem}\label{theorem_main}
For every $n \geq 2$ and $m<2^n$, there is a finite sequence $\vec{R} = \sequence{R_0,\ldots,R_{n-1}}$ such that for all partitions $A_0,\ldots,A_{m-1}$ of $\omega$ arithmetical in $\vec{R}$, there is an infinite subset of some $A_j$ that computes no $\vec{R}$-cohesive set.
\end{theorem}

\noindent The proof can then be iterated for different values of $n$ to also yield:

\begin{theorem}\label{theorem_general}
There is a sequence $\vec{R} = \sequence{R_i : i \in \omega}$ such that for all $m$ and all partitions $A_0,\ldots,A_{m-1}$ of $\omega$ arithmetical in $\vec{R}$, there is an infinite subset of some $A_j$ that computes no $\vec{R}$-cohesive set.
\end{theorem}

\noindent Thus, in particular, there is no way to arithmetically transform a sequence $\vec{R}$ into a set $A$, such that every infinite subset of $A$ or $\overline{A}$ computes an $\vec{R}$-cohesive set. This is surprising, as it is easy to build an $\vec{R}$-cohesive set that is $\Delta^0_3$ (and hence certainly arithmetical) in $\vec{R}$, and clearly every infinite subset of a cohesive set is also cohesive. It follows that cohesive sets can have ``cohesive avoiding'' sets in their complements.

Our work is motivated in part by questions in reverse mathematics, where cohesive sets have been useful in the analysis of the logical strength of various principles related to Ramsey's theorem for pairs. We refer the reader to Simpson \cite{Simpson-2009} for background on reverse mathematics, and to Hirschfeldt \cite{Hirschfeldt-TA} for a survey of reverse mathematical results concerning combinatorial principles.

\begin{definition}
For a set $S$, let $[S]^2$ denote $\{ \{x,y\} : x,y \in S\}$. A coloring $f : [\omega]^2 \to k$ is \emph{stable} if for every $x$, $\lim_y f(x,y)$ exists. Fix $k, n \in \omega$; the following statements are defined in $\RCA$:
\begin{enumerate}
\item $\COH$: Every sequence has a cohesive set.
\item $\RT^2_m$: Every $f : [\N]^2 \to m$ has an infinite homogeneous set, i.e., a set $H$ such that $f \res [H]^2$ is constant.
\item $\SRT^2_m$: Every stable $f : [\N]^2 \to m$ has an infinite homogeneous set.
\item $\D^2_m$: For every set $X$, if $A_0,\ldots,A_{m-1}$ are a partition of $\N$ definable by a formula $\Delta^0_2$ in $X$, then there exists an infinite subset of some $A_j$.
\end{enumerate}
\end{definition}

\noindent The principles are related as follows. The first theorem is essentially just due to the limit lemma, but proving it with $\Sigma^0_1$ induction alone is not straightforward.

\begin{theorem}[Chong, Lempp, and Yang \cite{CLY-2010}, Theorem 1.4] For each $m$, $\SRT^2_m$ is equivalent to $\D^2_m$ over $\RCA$.
\end{theorem}

\begin{theorem}[Cholak, Jockusch, and Slaman \cite{CJS-2001}, Lemma 7.11]
For each $m$, $\RT^2_m$ is equivalent to $\SRT^2_m + \COH$ over $\RCA$.
\end{theorem}

A long-standing open question, recently answered in the negative by Chong, Slaman, and Yang \cite{CSY-TA}, asked whether $\SRT^2_m$ implies $\COH$ over $\RCA$, or equivalently, whether $\D^2_m$ implies $\COH$. The separation in \cite{CSY-TA} is via a highly customized model with non-standard first-order part, leaving open the possibility that the implication still holds in $\omega$-models. The results of this paper can be viewed as saying that one of the standard methods for establishing this implication fails, as explained below.

Most principles studied in reverse mathematics are $\Pi^1_2$ as statements of second-order arithmetic, and as such, each has a natural class of \emph{instances}, and a natural class of \emph{solutions} to those instances. For example, the instances of $\COH$ are sequences of sets, and the solutions to a given instance $\vec{R}$ are the $\vec{R}$-cohesive sets.

\begin{definition}\label{definition_reductions}
Let $\mathsf{P}$ and $\mathsf{Q}$ be $\Pi^1_2$ statements of second-order arithmetic.
\begin{enumerate}
\item $\mathsf{P}$ is \emph{reducible} to $\mathsf{Q}$ if every instance $A$ of $\mathsf{P}$ computes an instance $B$ of $\mathsf{Q}$ such that for every solution $S$ to $B$, $S \oplus A$ computes a solution to $A$.
\item $\mathsf{P}$ is \emph{strongly reducible} to $\mathsf{Q}$ if every instance $A$ of $\mathsf{P}$ computes an instance $B$ of $\mathsf{Q}$ such that every solution to $B$ computes a solution to $A$.
\end{enumerate}
\end{definition}

\noindent A typical implication $\mathsf{Q} \to \mathsf{P}$ in $\omega$-models holds because there is a reduction of $\mathsf{P}$ to $\mathsf{Q}$, and in most cases it is actually a strong reduction. In fact, it is common for every solution to the instance $B$ of $\mathsf{Q}$ to be itself a solution to the instance $A$ of $\mathsf{P}$. For example, this is the case for the implication from $\RT^2_m$ to $\COH$ (see Mileti \cite{Mileti-2004}, Corollary A.1.4), and for most other previously studied implications between combinatorial principles, e.g., in \cite{HS-2007}, \cite{HSS-2009}, and \cite{DM-TA}.

In the parlance of Definition \ref{definition_reductions}, Theorems \ref{theorem_main} and \ref{theorem_general} then have the following consequence:

\begin{corollary}
For each $m$, $\COH$ is not strongly reducible to $\D^2_m$. In fact, $\COH$ is not strongly reducible to the principle $\forall m\,\D^2_m$.
\end{corollary}

We present the proofs of the theorems below. The main ingredient is Lemma \ref{lemma_combinatorial}, which uses an elaboration on the method of Seetapun for proving that every computable coloring of pairs has an infinite homogeneous set that does not compute $\emptyset'$ (see \cite{SS-1995}, Theorem 2.1). A simpler proof of Seetapun's result is given in \cite{DJ-2009}, but it is less clear how to adapt that argument for our purposes.

\section{Proofs of Theorems \ref{theorem_main} and \ref{theorem_general}}

We assume familiarity with the basics of forcing in arithmetic, as outlined, e.g., in \cite[Chapter 3]{Shore-TA}. Below, we use the symbol $\forces$ to refer to the forcing relation of both Cohen forcing and the forcing notion defined in Definition \ref{definition_forcing}. Throughout, \emph{generic} will mean arithmetically generic, i.e., $n$-generic for all $n \in \omega$.

\subsection{Proof of Theorem \ref{theorem_main}}

Fix $n \geq 2$ and $m < 2^n$. We may assume $m \geq 2$, as the only partition of $\omega$ into one set is $\omega$ alone, and then it suffices to take $\omega$ as the desired subset and add any bi-immune set to $\vec{R}$. Let $\Gamma_{i,j}$ for $i \in \omega$ and $j<m$ range over all $m$-tuples of arithmetical operators. The sequence $\vec{R}$ is obtained generically for the following notion of forcing:

\begin{definition}\label{definition_forcing}
\Needspace*{3\baselineskip}%
\
\begin{enumerate}
\item A condition is a sequence of the form
\[
p = \sequence{\sigma^p, \sequence{F^p_{i,0}, \ldots, F^p_{i,m-1}, L^p_i} : i \in \omega},
\]
such that:
\begin{enumerate}
\item $\sigma^p$ is a finite binary string;
\item $F^p_{i,0}, \ldots, F^p_{i,m-1}$ are finite sets;
\item for all $j < m$ and all $x \in F^p_{i,j}$, $\sigma^p \forces x \in \Gamma^G_{i,j}$;
\item $L^p_i$ is an infinite set with $\max \bigcup_{j < m} F^p_{i,j} < \min L^p_i$.
\end{enumerate}
\item  A condition $q$ extends $p$ if $\sigma^p \preceq \sigma^q$, and for each $i$ and $j < m$, $L^q_i \subseteq L^p_i$ and $F^p_{i,j} \subseteq F^q_{i,j} \subset F^p_{i,j} \cup L^p_i$.
\end{enumerate}
\end{definition}

Each condition can thus be regarded as a Cohen condition, followed by a sequence of what are essentially Mathias conditions. (See, e.g., \cite{DJ-2009}, Definition 2.1 for a definition of Mathias forcing.) From a generic filter $\mathcal{F}$ for this notion of forcing, we obtain a set $G$ as the union of the $\sigma^p$ for $p \in \mathcal{F}$. For each $i$, we also obtain sets $G_{i,\0},\ldots,G_{i,m-1}$ as the unions of the $F^p_{i,0}, \ldots, F^p_{i,m-1}$, respectively, for $p \in \mathcal{F}$. We view $G$ as an $n$-fold join, and let $\vec{R}$ consist of the columns; that is, $G = R_0 \oplus \cdots \oplus R_{n-1}$.

\begin{lemma}\label{lemma_subset}
For each $i \in \omega$ and $j < m$, $G_{i,j} \subseteq \Gamma^G_{i,j}$.
\end{lemma}

\begin{proof}
This is immediate from the definition of condition. Given $x \in G_{i,j}$, there is a condition $p$ in $\mathcal{F}$ with $x \in F^p_{i,j}$. Then $\sigma^p \forces x \in \Gamma^G_{i,j}$, and as $\sigma^p$ is an initial segment of $G$, which is clearly Cohen generic, this implies that $x$ belongs to $\Gamma^G_{i,j}$.
\end{proof}

In the sequel, all computations from a finite oracle $F$ will be assumed to have use bounded by $\max F$. The main component of our argument is the following combinatorial lemma, whose proof we leave for the next section. Given a condition $p$ and numbers $i$ and $j$, say a set $S$ \emph{satisfies} $\sequence{F^q_{i,j},L^q_i}$ if $F^q_{i,j} \subseteq S \subseteq F^q_{i,j} \cup L^q_i$.

\begin{lemma}\label{lemma_combinatorial}
Let $i \in \omega$, $u_0,\ldots,u_{n-1} \in \{0,1\}$, and reductions $\Delta_0,\ldots,\Delta_{m-1}$ be given. Let $p$ be a condition forcing that $\Gamma^G_{i,0},\ldots,\Gamma^G_{i,m-1}$ partition $\omega$. Then there is a $q \leq p$ and a $j < m$ such that $p \not\forces G_{i,j} \text{ is finite}$ and one of the following is true:
\begin{enumerate}
\item if $S$ is any set satisfying $\sequence{F^q_{i,j},L^q_i}$, then $\Delta_j^S$ is not an infinite set;
\item there is a $y > |\sigma^p|$ such that $\Delta_j^{F^q_{i,j}}(y) \downarrow = 1$ and $R_k(y) = u_k$ for all $k < n$.
\end{enumerate}
\end{lemma}

To complete the proof of Theorem \ref{theorem_main}, fix $i \in \omega$ and suppose $\Gamma^G_{i,0},\ldots,\Gamma^G_{i,m-1}$ partition $\omega$. By Lemma \ref{lemma_subset}, to exhibit an infinite subset of one of the $\Gamma^G_{i,j}$ that computes no $\vec{R}$-cohesive set, it suffices to prove the following:

\begin{lemma}\label{lemma_end}
One of the sets $G_{i,0},\ldots,G_{i,m-1}$ is infinite and computes no $\vec{R}$-cohesive set.
\end{lemma}

\begin{proof}
First, suppose $G_{i,0},\ldots,G_{i,m-1}$ are all finite. Then there is a number $s$ and a condition $p$ in the generic filter $\mathcal{F}$ forcing that $G_{i,0},\ldots,G_{i,m-1}$ are bounded by $s$. So, for any $x > s$ in $L^p_i$, no extension of $\sigma^p$ can force $x \in \Gamma^G_{i,j}$ for any $j < m$, as otherwise we could define a condition $q \leq p$ with $x \in F^q_{i,j}$. But this means no such $\sigma$ can force that the $\Gamma^G_{i,j}$ partition $\omega$, which is a contradiction. Thus, at least one of $G_{i,0},\ldots,G_{i,m-1}$ is infinite, as wanted.

Now fix reductions $\Delta_0,\ldots,\Delta_{m-1}$ such that $\Delta_j^{G_{i,j}}$ is an infinite set whenever $G_{i,j}$ is. Given $s \in \omega$ and an $n$-tuple $u_0,\ldots,u_{n-1} \in \{0,1\}$, let $\mathcal{C}$ be the set of all conditions $p$ such that for some $j<m$ with $G_{i,j}$ infinite, there is a $y > s$ such that $\Delta_j^{F^p_{i,j}}(y) \downarrow = 1$ and $R_k(y) = u_k$ for all $k < n$. Lemma \ref{lemma_combinatorial} implies that $\mathcal{C}$ is dense in $\mathcal{F}$, so $\mathcal{C} \cap \mathcal{F} \neq \emptyset$. Since, for any $p \in \mathcal{C} \cap \mathcal{F}$ and any $j < m$, $G_{i,j}$ satisfies $\pair{F^p_{i,j},L^p_i}$, it follows that for some $j<m$ with $G_{i,j}$ infinite, there is a $y > s$ such that $\Delta_j^{G_{i,j}}(y) \downarrow = 1$ and $R_k(y) = u_k$ for all $k < n$.

It follows that for each $n$-tuple $u_0,\ldots,u_{n-1} \in \{0,1\}$, we can pick a $j < m$ with $G_{i,j}$ infinite such that for infinitely many $y$, $\Delta_j^{G_{i,j}}(y) \downarrow = 1$ and $R_k(y) = u_k$ for all $k < n$. But as $m < 2^n$, there must be two different such $n$-tuples for which we can pick the same $j$. Thus, for some $k < n$, $\Delta_j^{G_{i,j}}$ intersects both $R_k$ and $\overline{R_k}$ infinitely often, and hence $\Delta_j^{G_{i,j}}$ is not cohesive for $R_k$. As $\Delta_0,\ldots,\Delta_{m-1}$ were arbitrary, we conclude that some infinite $G_{i,j}$ computes no $\vec{R}$-cohesive set.
\end{proof}

\subsection{Proof of Theorem \ref{theorem_general}} 

We now build an infinite sequence $\vec{R} = \sequence{R_i : i \in \omega}$, and for each $m \geq 2$ reserve some number $n = n_m$ of columns of $\vec{R}$ with which to carry out the above proof. (The precise value of $n$ is inconsequential so long as $m < 2^n$.) Thus we are essentially just folding together versions of the above argument for different values of $n$ and $m$. The only difference from the ``local'' construction above is the presence of additional columns in $\vec{R}$, but as each step of the construction deals with only one particular $m$, these can always be extended arbitrarily when necessary without complication.

More precisely, for each $m$, let $\Gamma^m_{i,j}$ for $i \in \omega$ and $j<m$ range over all $m$-tuples of arithmetical operators. We modify our conditions to be sequences
\[
p = \sequence{\sigma^p, \sequence{F^{p,m}_{i,0}, \ldots, F^{p,m}_{i,m-1}, L^{p,m}_i} : m \geq 2, i \in \omega},
\]
such that for all $x \in F^{p,m}_{i,j}$, $\sigma^p \forces x \in \Gamma^{m}_{i,j}$, with the rest of the clauses, and extension, defined as in Definition \ref{definition_forcing} with the obvious modifications. The generic objects are now the set $G$, whose columns as an $\omega$-fold join serve as the members of $\vec{R}$, and for all $n,m,i \in \omega$, sets $G^m_{i,0},\ldots,G^m_{i,m-1}$. For these, analogues of Lemmas \ref{lemma_subset}, \ref{lemma_combinatorial}, and \ref{lemma_end} are easily formulated, in the case of the latter two simply by replacing all references to $\vec{R}$ with the $n_m$ many columns of $G$ reserved for $m$. \emph{Mutatis mutandis}, the lemmas can be proved by virtually the same arguments as before. We leave the details to the reader.

\section{Proof of Lemma \ref{lemma_combinatorial}}

Throughout, we assume we have fixed $i \in \omega$, $u_0,\ldots,u_{n-1} \in \{0,1\}$, reductions $\Delta_0,\ldots,\Delta_{m-1}$, and a condition $p$ as in the statement of the lemma. By passing to a stronger condition if necessary, we may assume that for all $j < m$, if $p \forces G_{i,j} \text{ is finite}$ then $p \forces (\forall x > \min L^p_i) [x \notin G_{i,j}]$. For convenience, by re-defining $m$ and re-labeling, we may also assume the indices $j < m$ in the rest of the proof range only over those $j$ for which $p \not\forces G_{i,j} \text{ is finite}$.

We shall make use of the following notions in the proof:

\begin{definition}
Fix $j < m$. For a tree $T \subseteq \omega^{<\omega}$, let $\range(T)$ denote $\bigcup_{\tau \in T} \range(\tau)$.
\begin{enumerate}
\item A \emph{$\Delta_j$-sequence} is a sequence $T_0,T_1,\ldots$ of finite subtrees of $(L^p_i)^{<\omega}$ such that $\max \range(T_k) < \min \range(T_{k+1})$ for all $k$, and if $\tau$ is a terminal node of $T_k$ there is a finite $F \subset \range(\tau \res |\tau| - 1)$ and a $y > |\sigma^p|$ such that $\Delta_j^{F^p_{i,j} \cup F}(y) \downarrow = 1$.
\item If $T_0,T_1,\ldots$ is a $\Delta_{j-1}$-sequence, then its \emph{$\Delta_j$-tree} is the set of all $\tau \in \omega^{<\omega}$ such that $\tau(k) \in \range(T_k)$ for all $k < |\tau|$, and there is no finite $F \subset \range(\tau \res |\tau| - 1)$ and no $y > |\sigma^p|$ such that $\Delta_j^{F^p_{i,j} \cup F}(y) \downarrow = 1$.
\end{enumerate}
\end{definition}

\noindent The key point of this definition is that if some $\Delta_{j-1}$-sequence has a finite $\Delta_j$-tree, then the range of each terminal node of this tree has a finite subset $F$ such that $\Delta_j^{F^p_{i,j} \cup F}(y) \downarrow = 1$ for some $y > |\sigma^p|$. (This is the reason for looking at $\range(\tau \res |\tau|-1)$ above, instead of at $\range(\tau)$ as may seem more natural.) Moreover, the finite $\Delta_j$-tree can serve as the starting point of a $\Delta_j$-sequence.

The proof divides into three cases.

\begin{case}{1}
For some $j$, there is no $\Delta_j$-sequence.
\end{case}

\noindent There must exist an $s$ such that there is no finite $F \subset L^p_i$ with $\min F > s$, and no $y > |\sigma^p|$, such that $\Delta^{F^p_{i,j} \cup F}_j(y) \downarrow = 1$. Otherwise, by identifying each such finite set $\{x_0,\ldots,x_{s-1}\}$ with the tree of all initial segments of the string $x_0 \cdots x_{s-1}$, we could build a $\Delta_j$-sequence. So let $q$ be exactly the condition $p$, except for the modification $L^q_i = \{x \in L^p_i : x > s\}$. Then if $S$ is any set satisfying $\sequence{F^q_{i,\In}, L^q_i}$, it cannot be that $\Delta^S_j(y) \downarrow = 1$ for any $y > |\sigma^p|$, so $\Delta^S_j$ is not infinite.

\begin{case}{2}
For some $j < m$, there is a $\Delta_{j-1}$-sequence with infinite $\Delta_j$-tree.
\end{case}

\noindent Let $T$ be the infinite $\Delta_j$-tree, and consider any infinite path through $T$. Let $q$ be exactly the condition $p$, except for the modification that $L^q_i$ is the range of this path. Since the path is an increasing sequence, $q$ is a condition. Then if $S$ is any set satisfying $\sequence{F^p_{i,j}, L^q_i}$, it cannot be that $\Delta^S_j(y) \downarrow = 1$ for any $y > |\sigma^p|$, so $\Delta^S_j$ is not infinite.

\begin{case}{3}
Otherwise.
\end{case}

\noindent For each $j < m$, we inductively define a $\Delta_j$ sequence $T_{j,0}, T_{j,1}, \ldots$ as follows. Let $T_{0,0}, T_{0,1}, \ldots$ be a $\Delta_0$-sequence, as exists by the failure of Case 1. Fix $j < m-1$ and $k \in \omega$, and assume that $T_{j,0},T_{j,1},\ldots$ have been defined, along with $T_{j+1,k'}$ for all $k' < k$. Choose the least $s$ such that $\max \bigcup_{k' < k} \range(T_{j+1,k'}) < \min \range(T_{j,s})$, or equivalently, let $s$ be the sum of the heights of the $T_{j+1,k'}$ for $k' < k$. Then let $T_{j+1,k}$ be the $\Delta_{j+1}$-tree of the increasing tree sequence $T_{j,s},T_{j,s+1},\ldots$, which is finite by the failure of Case 2. Clearly, $T_{j+1,0},T_{j+1,1},\ldots$ so defined is a $\Delta_{j+1}$-sequence.

We next define numbers $s_0,\ldots,s_{m-1}$ by reverse induction. Let $s_{m-1} = 1$, and having defined $s_j$ for some non-zero $j < m$, let $s_{j-1}$ be the least $s$ such that $\bigcup_{k < s_j} \range(T_{j,k}) \subseteq \bigcup_{k < s} \range(T_{j-1,k})$. Thus, essentially, $s_{j-1}$ picks out the $T_{j-1,k}$ relevant to the construction of $T_{j,0},\ldots,T_{j,s_j-1}$.

By definition, if $\tau$ is a terminal node of some $T_{j,k}$, then there is a finite $F \subset \range(\tau)$ and a $y > |\sigma^p|$ such that $\Delta_j^{F^p_{i,j} \cup F}(y) \downarrow = 1$. Let $z$ be the maximum of all such $y > |\sigma^p|$ for the $T_{j,k}$ with $j < m$ and $k < s_j$. Let $\sigma \in 2^{<\omega}$ be the extension of $\sigma^p$ of length $nz+n$ defined for $|\sigma_p| \leq k < nz+n$ by $\sigma(b) = u_k$ if $b \equiv k \mod n$. In other words, viewing $\sigma^p$ as an $n$-fold join, we extend its $k$th column to equal $u_k$ on all numbers up to $z$.

Now since $\sigma^p$ forces that $\Gamma^G_{i,0},\ldots,\Gamma^G_{i,m-1}$ partition $\omega$, we may fix an extension $\sigma'$ of $\sigma$ that, for each $x \in \bigcup_{j < m} \bigcup_{k < s_j} \range(T_{j,k})$, forces $x \in \Gamma^G_{i,j}$ for some $j$. We claim there is a $j < m$, a $k < s_j$, and a terminal node $\tau \in T_{j,k}$ such that $\sigma' \forces x \in \Gamma^G_{i,j}$ for all $x \in \range(\tau)$. If so, we can choose a finite $F \subseteq \range(\tau)$ such that $\Gamma^{F^p_{i,j} \cup F}(y) \downarrow = 1$ for some $y > |\sigma^p|$, and let $q$ be exactly the condition $p$, except for the modifications $\sigma^q = \sigma'$, $F^q_{i,j} = F^p_{i,j} \cup F$, and $L^q_i = \{x \in L^p_i : x > \max F\}$. Then $q$ is a condition, and by construction there is a $y > |\sigma^q|$ such that $\Gamma^{F^q_{i,j}}(y) \downarrow = 1$ and $R_k(y) = \sigma(ny+k) = u_k$ for all $k < n$.

To prove the claim, suppose there are no such $j$, $k$, and $\tau$ with $j < m-1$; we prove the claim holds for $j = m-1$ and $k = 0$. By choice of $\sigma'$, it suffices to show that for all $j < m$ and $k < s_j$, $T_{j,k}$ has a terminal node with range consisting entirely of numbers $x$ such that $\sigma' \forces x \notin \Gamma^G_{i,l}$ for all $l < j$. This clearly holds for $j = 0$, so suppose it also holds for some $j < m-1$. By hypothesis, if $\tau$ is the appropriate terminal node of $T_{j,k}$ for some $k < s_j$, then $\range(\tau)$ contains an $x$ such that $\sigma' \forces x \notin \Gamma^G_{i,j}$, and so $\sigma' \forces x \notin \Gamma^G_{i,l}$ for all $l < j+1$. The claim now follows for $j +1$ by the definition of the $T_{j+1,k}$ and of the sequence $s_0,\ldots,s_{m-1}$.


\end{document}